\documentclass[11pt, oneside, dvipsnames, svgnames, table, final]{amsart}

\usepackage[T1]{fontenc}
\usepackage[utf8]{inputenc}

\usepackage{fourier}

\usepackage[backref=true,style=alphabetic,citestyle=alphabetic,url=false]{biblatex}
\usepackage[]{tbjwHeader}
\usepackage[]{spectralsequences}

\newcommand{\rM}{\mathrm{M}}

\newcommand{\sebg}[2][]{\ifdraft{\todo[linecolor=Red,backgroundcolor=Red!25,bordercolor=Red,#1]{#2---Seb G.}}{}}
\newcommand{\benw}[2][]{\ifdraft{\todo[linecolor=Green,backgroundcolor=Green!25,bordercolor=Green,#1]{#2---Ben W.}}{}}

\begin{document}   

\title{Spaces of generators for the $2 \times 2$ complex matrix algebra}

\author{W. S. Gant} 
\address{W. S. Gant. Department of Mathematics, University of British Columbia, Vancouver~BC V6T~1Z2, Canada.}
\email{wsgant@math.ubc.ca}

\author{Ben Williams} \thanks{We acknowledge the support of the Natural Sciences and Engineering Research Council of
  Canada (NSERC), RGPIN-2016-03780, RGPIN-2021-02603. \\ Cette recherche a \'et\'e financ\'ee par le Conseil de recherches en sciences naturelles et en g\'enie du Canada (CRSNG),  RGPIN-2016-03780, RGPIN-2021-02603.}
\address{Ben Williams. Department of Mathematics, University of British Columbia, Vancouver~BC V6T~1Z2, Canada.}
\email{tbjw@math.ubc.ca}


\begin{abstract}
  This paper studies $B(r)$, the space of $r$-tuples of $2 \times 2$ complex matrices that generate
  $\Mat_{2 \times 2}(\C)$ as an algebra, considered up to change-of-basis. We show that $B(2)$ is homotopy equivalent to
  $S^1 \times^{\Z/2\Z} S^2$. For $r>2$, we determine the rational cohomology of $B(r)$ for degrees less than $4r-6$. As
  an application, we use the machinery of \cite{First2022} to prove that for all natural numbers $d$, there exists a
  ring $R$ of Krull dimension $d$ and a degree-$2$ Azumaya algebra $A$ over $R$ that cannot be generated by fewer than
  $2\lfloor d/4 \rfloor + 2$ elements.
\end{abstract}
\maketitle


\section{Introduction} \label{sec:Introduction}

Throughout, $r \ge 2$ denotes a natural number. This paper is an investigation of the topology of two families of
spaces, denoted $U(r)$ and $B(r)$ below, related to generating $r$-tuples for the $2 \times 2$ matrix algebra.

Let $U(r)$ denote the space of $r$-tuples of $2 \times 2$ matrices
that generate $\Mat_{2 \times 2}(\C)$ as an algebra. By a theorem of Burnside
(\cite{Lomonosov2004}) this is the set of $r$-tuples $(A_1, \dots, A_r)$ that do not have an
eigenvector in common.

The group $\PGL(2; \C)$ acts on $\Mat_{2 \times 2}(\C)$ by conjugation, in other words, by change of basis. It acts on
$U(r)$ by simultaneous conjugation. This action is free, see Proposition \ref{pr:free}, and
the quotient $B(r)$ is a manifold, see Proposition \ref{pr:quotientIsPrinBundle}.

The spaces $B(r)$ are approximations to the classifying space $B \PGL(2)$, in the pattern of 
\cite{Totaro1999} or \cite{Edidin1998}.  The map $U(r) \to B(r)$ is a
principal $\PGL(2;\C)$-bundle and is therefore classified by a map $B(r) \to B\PGL(2;\C)$, unique up to homotopy. The
classifying map is topologically $(2r-3)$-connected (Corollary \ref{cor:connectivityOfClassMap}), so that in particular, the induced map on
cohomology $\Hoh^i (B \PGL(2;\C); \Q) \to \Hoh^i(B(r); \Q)$ is an isomorphism in the range $i \le 2r-4$. This paper
examines the invariants of $B(r)$ beyond this range.

Three of our four main results are Proposition \ref{pr:B2type}, which determines the homotopy type of $B(2)$ completely, and
Propositions \ref{pr:cohoBOdd} and \ref{pr:cohoBEven}, which calculate the rational cohomology $\Hoh^i(B(r); \Q)$ when $i
< 4r-6$.

\subsection*{Applications to generation of algebras}

Our last main result is to answer a question left open by \cite{First2022}, which we now explain.

Over a commutative ring $R$, an \textit{Azumaya algebra of degree $2$} is a unital, associative but not commutative $R$-algebra
$A$ such that there exists a faithfully flat extension $\phi: R \to S$ such that $A \tensor_R S$ is isomorphic to
$\Mat_{2 \times 2}(S)$ as an $S$-algebra. Associated to such an algebra $A$ over a ring $R$, there is an algebraic principal $\PGL(2)$-bundle $T \to\Spec
R$ given locally on $\Spec R$ as the bundle of isomorphisms $A|_U \to \Mat_{2 \times 2} \times U$. As a special case of \cite[Proposition 4.1]{First2022} applied to $\Mat_{2 \times 2}$, there is a natural
bijective correspondence:
\begin{equation}
  \left\{\parbox{6.6cm}{$\PGL(2)$-equivariant morphisms $T \to
      U(r)$}\right\} \leftrightarrow \left\{\parbox{9cm}{$r$-tuples of elements of $A$ that generate
      $A$ as an algebra.}\right\}.\label{eq:4}
\end{equation}

If $A$ and $A'$ are two algebras over $R$ and $(a_1, \dots, a_r) \in A^r$ and $(a'_1, \dots, a_r') \in (A')^r$ are
$r$-tuples of elements, then we will say that an \emph{isomorphism} $\phi: (A; a_1, \dots, a_r) \to (A'; a_1', \dots,
a_r')$ is an isomorphism $\phi : A \to A'$ so that $\phi(a_i) = a_i'$ for all $i \in \{1, \dots, r\}$.

By passing to quotients, we arrive at:
  \begin{proposition}
    With notation as above, there exists a bijective correspondence
    \begin{equation}
  \left\{\parbox{4.2cm}{Morphisms $\Spec R \to
      B(r)$}\right\} \leftrightarrow \left\{\parbox{9cm}{Isomorphism classes of Azumaya algebras of degree $2$, equipped
      with $r$-tuples of elements of $A$ that generate
      $A$ as an algebra.}\right\}.\label{eq:8}
\end{equation}
\end{proposition}
\begin{proof}
  First we relate the left-hand side of \eqref{eq:8} with that of \eqref{eq:4}. In brief, the variety $B(r)$ is the
  stack quotient of $U(r)$ by $\PGL(2)$. More explicitly, given a $\PGL(2)$-equivariant morphism $T \to U(r)$, take the
quotient to arrive at a morphism $\Spec R \to B(r)$. Conversely, given a morphism $f: \Spec R \to B(r)$, one can pull
the torsor $U(r) \to B(r)$ back along $f$ to produce a torsor $T \to \Spec R$, defined up to isomorphism. These
operations are inverse to each other, provided the torsor is always considered up to isomorphism, so we see that
the set of morphisms $\Spec R \to B(r)$ is in bijective correspondence with the set of isomorphism classes of spans
\[ \Spec R \leftarrow T \rightarrow U(r) \]
where the left map is a $\PGL(2)$-torsor and the right map is $\PGL(2)$-equivariant.

That is, the left-hand
side of \eqref{eq:8} is in bijective correspondence with isomorphism classes of possible left-hand sides of
\eqref{eq:4} as the torsor is allowed to vary. Naturality of the correspondence \eqref{eq:4} implies that the left-hand
side of \eqref{eq:8} is in bijective correspondence with isomorphism classes of possible right-hand sides of
\eqref{eq:4} as the algebra is allowed to vary. This is what we wanted.
\end{proof}

If we specialize to the case where $R$ is a reduced finite-type $\C$-algebra, the objects $T$ and $\Spec R$ are complex
varieties, and we may use topological methods to give computable obstructions to generating $A$ by $r$ elements. Write
$X$ for $\Spec R$ viewed as a complex analytic space. We may find a topological obstruction to the existence of
$\PGL(2; \C)$-equivariant maps $T \to U(r)$, by finding obstructions to a lift in the diagram:
\begin{equation} \label{eq:6} 
  \begin{tikzcd}
    &  B(r) \arrow[d] \\ X \arrow[r, "f"'] \arrow[dashed, ur]  & B \PGL(2;\C)  
  \end{tikzcd}
\end{equation}
where the maps to $B \PGL(2;\C)$ are those classifying the $\PGL(2;\C)$-bundles $T\to X$ and $U(r) \to B(r)$. From this point of
view, any cohomology class $\alpha \in \Hoh^*(B \PGL(2;\C); \ZZ)$ that maps to $0$ in $\Hoh^*(B(r); \ZZ)$ furnishes a
potential obstruction for $A$ to be generated by $r$ elements, and the obstruction actually applies if $f^*(\alpha) \neq 0$.

The final sections of \cite{First2022} are concerned with using this idea to produce examples of algebras of various kinds that cannot be
generated by $r$ elements. In \cite[Sections 13 and 14]{First2022}, the calculations for forms of
$\Mat_{s \times s}(\C)$ are done when $s > 2$, i.e., for Azumaya algebras of degree at least $3$. The theorem of \cite{First2022} is
essentially\footnote{The cited theorem applies over all fields of characteristic $0$.} this:
\begin{theorem}[Theorem 1.5(b) of \cite{First2022}]
  For any $s > 2$ and any integer $d \ge 0$, there exists a finite type ring $R$ over $\C$ of Krull dimension $d$ and an
  Azumaya $R$-algebra $B$ such that $B$ cannot be generated by fewer than
  \[ \left\lfloor \frac{d} {2(s-1)} \right\rfloor + 2 \]
  elements.
\end{theorem}

The proof of \cite{First2022} does not extend to the case of $s=2$. The spaces $U(r)$ and $B(r)$ behave differently from
their analogues for $s>2$. The techniques of \cite{First2022} cannot determine the smallest $i$ such
that the natural map $\Hoh^i(B\PGL(2;\C); \QQ) \to \Hoh^i(B(r); \QQ)$ is not injective, in contrast to the case of
higher $i$ which is handled by \cite[Lemma 14.5]{First2022}.

In this paper, we adopt different methods to study $U(r)$, $B(r)$. Provided $r> 2$, we show that the natural map
\begin{equation*}
  \Hoh^i(B \PGL(2;\C); \QQ) \to \Hoh^i(B(r); \QQ)
\end{equation*}
is not injective when $i = 2r-2$ and $r$ is odd, or when $i = 2r$ and $r$ is even. These statements follow from
Propositions \ref{pr:cohoBOdd} and \ref{pr:cohoBEven}, along with the observation that $\Hoh^*(B \PGL(2;\C); \QQ) =
\QQ[p_1]$ where $|p_1| = 4$.

The same argument used to prove \cite[Theorem 1.5(b)]{First2022} now establishes our last main result:
\begin{theorem} \label{th:fakeMainTheorem} If $d \ge 0$, there exists a finitely generated $\C$-ring $R$ of Krull
  dimension $d$ and an Azumaya algebra $A$ of degree $2$ over $R$ such that $A$ cannot be generated by fewer than
  \[ 2 \left \lfloor \frac{d}{4} \right \rfloor + 2\]
  elements.
\end{theorem}

This theorem represents an improvement in our understanding of how many generators are required for a general degree-$2$
Azumaya algebra $A$ over a ring $R$ of Krull dimension $d$. In \cite[Theorem 1.2]{First2017}, it is proved that all such
$A$ can be generated by $d+2$ elements. 

An alternative argument, not involving the spaces $U(r)$ and $B(r)$, in \cite[Sec.~13]{First2022},
uses \cite{Shukla2020} to produce examples of finitely generated $\RR$-rings $R$ of Krull dimension $d$ and
quaternion algebras $A$ over $R$ that require $\lfloor \frac{d+1}{2} \rfloor$ generators. The
examples we produce in this paper represent an improvement on this in two respects.

First,
\[ 2 \left\lfloor \frac{d}{4} \right\rfloor + 2 \ge \left\lfloor \frac{d+1}{2} \right\rfloor \] with equality only when
$d \equiv 3 \pmod 4$, so that our result is a numerical improvement on that of
\cite[Sec.~13]{First2022}.

Second, the examples of
\cite[Sec.~13]{First2022} rely implicitly on the fact that $-1$ is not a square in $R$. The examples
we give here can be constructed over any field of characteristic $0$, and the same bounds on the
number of generators required can be obtained from the bounds over $\CC$ by use of the Lefschetz
principle, see \cite[Sec.~10]{First2022}. We give the arguments over $\CC$ to keep the paper short.

\medskip

Beyond proving Theorem \ref{th:fakeMainTheorem}, the topology of $B(r)$ and its relation to $B \PGL(2;\C)$ can answer other
questions about whether or not Azumaya algebras of degree $2$ can be generated by $r$ elements, by virtue of the lifting
obstruction in \eqref{eq:6}. 


\subsection*{Acknowledgements}

We are extremely grateful to Zinovy Reichstein for many helpful conversations and pointers to the literature in this
area. The paper has benefitted greatly from his generous advice.

\section{Preliminaries}

Algebras in this paper are unital, associative $\C$-algebras, not necessarily commutative. Here are some conventions
that hold in the rest of the paper.

\benw[inline]{Put definitions and common notation here}
\begin{itemize}
\item $\rM$ denotes $\Mat_{2 \times 2}(\C)$.
\item Except where it would lead to confusion, we write $\PGL(2)$ for the complex Lie group $\PGL(2;\C)$.
\item $r$ is a positive integer, and to avoid vacuity, $r \ge 2$.
\item An $r$-tuple of matrices $(A_1, \dots, A_r) \in \rM^r$ may be denoted $\vec{A}$.
\item $Z(r)$ is the subset of $\rM^r$ consisting of $r$-tuples $(A_1, \dots, A_r)$ that generate a proper subalgebra of
  $\rM$. 
\item $U(r)$ is the complement of $Z(r)$ in $\rM^r$.
\item The term ``variety'' will mean ``complex algebraic variety'' and all varieties we encounter are
  quasiprojective. By a ``point'' of a variety $X$, we mean a complex point.
\end{itemize}

\sebg[inline]{Need to include/provide references for: \\
  Connectedness of the complement of a closed subvariety of $\A^n$ \\ Thom-Gysin sequence }
\benw[inline]{We'll probably just assert these things and dare the referee to call us on them.}


\begin{proposition}
  The space $Z(r)$ of non-generating $r$-tuples is a closed subvariety of codimension $r-1$ in $\rM^r$.
\end{proposition}
\begin{proof}
  This is a special case of \cite[Proposition 7.1]{First2022}. Consider the incidence variety $Y$ in $\rM^r \times
  \CP^1$ consisting of pairs $\vec A \times L$ where $L$ is an eigenspace for each matrix in the $r$-tuple $\vec
  A$:
  \begin{equation*}
    \begin{tikzcd}
       & Y \arrow[dr, "\proj_2"] \ar[dl, "\proj_1"'] \\ \rM^r & & \CP^1.
    \end{tikzcd}
  \end{equation*}
  Since $\PGL(2;\C)$ acts on $Y$ and transitively on $\CP^1$, the projection $Y \to \CP^1$ is flat. Each fibre of the
  projection $Y \to \CP^1$ is an affine space of dimension $3r$, so that $Y$ has dimension $3r +1$. Since $Y$ is closed in $\rM^r \times \CP^1$, and $\CP^1$ is a proper variety, the image of $Y$ in $\rM^r$ is a closed
  subvariety. This image is precisely $Z(r)$.

  It remains to show that $\dim_\C Z(r) = \dim_\C Y$. There is a dense open subset $Y_0 \subset Y$ where the first
  matrix in $\vec A$ has at most $2$ different eigenspaces, and therefore the map $Y_0 \to \rM^r$ has finite fibres. It
  follows that the image of $\proj_1|_{Y_0}$ is a locally closed subvariety of $\rM^r$ of dimension $3r-1$. This
  image is dense in $Z(r)$ and so the result follows.
\end{proof}

There is an action of $\PGL(2)$ on $\rM^r$ given by
\begin{equation*}
  G \cdot (A_1, \dots, A_r) = (G A_1 G^{-1}, \dots, G A_r G^{-1}).
\end{equation*} This action restricts to an action of $\PGL(2)$ on $U(r)$.

Let $M'(r)$ denote the set of $r$-tuples of $2\times 2$-matrices $(A_1, \dots, A_r)$ such that $\Tr(A_i) = 0$ for all
$i$. There is a deformation retraction $\Phi$ of $\rM^r$ onto $M'(r)$, given by
\begin{equation}
  \label{eq:3} \Phi : \rM^r \times [0, 1] \to M, \qquad\Phi (A_1, \dots, A_r,t ) = (A_1 - t \Tr(A_1), \dots, A_r - t \Tr(A_r)).
\end{equation} Write $U'(r)$ for the intersection $U(r) \cap M'(r)$. The space $U'(r)$ is invariant under the $\PGL(2,
\C)$-action, and therefore there is a quotient $B'(r) = U'(r)/\PGL(2, \C)$.

\begin{proposition} \label{pr:tracelessHEquiv} The inclusions $U'(r) \to U(r)$ and $B'(r) \to B(r)$ are homotopy
equivalences.
\end{proposition}
\begin{proof} We observe that the deformation retraction in \eqref{eq:3} restricts to a deformation retraction of $U(r)$
onto $U'(r)$, and induces a deformation retraction of $B(r)$ onto $B'(r)$.
\end{proof}

\section{Invariant theory} \label{sec:InvTheory}

We wish to understand the action of the reductive group $\PGL(2)$ on $U(r)$ and in particular the quotient $B(r)=U(r)/
\PGL(2)$. What we present here is the $n=2$ case of a much more general theory, developed in \cite{Sibirskii1968}, \cite{Artin1969},
\cite{Procesi1976}, \cite{LeBruyn1987}, and elsewhere.

In this section, we encounter $\C$-schemes which we do not show to be varieties. While we are happy to use ``point of  $V$'' to
mean ``rational point of $V$'' when $V$ is a variety, whenever it is not obvious that $X$ is a variety,  we use the more
precise term ``rational  point''.

\begin{proposition} \label{pr:free}
  The action of $\PGL(2)$ on $U(r)$ is free.
\end{proposition}
\begin{proof}
  Let  $(A_1, A_2, \dots, A_r)$ be a point of $U(r)$.
  Suppose some $G \in \PGL(2)$ satisfies $GA_iG^{-1} = A_i$ for all $i$. Then conjugation by $G$ fixes any matrix that can be
  expressed as a polynomial in the $A_i$. Since this is the set of all matrices, we see that $G$ must act trivially on
  $\rM$. It follows that $G$ is the identity element of $\PGL(2)$. 
\end{proof}

Consider the coordinate ring of $r$-tuples of matrices, $\rM^r$. This ring is a polynomial ring in $4r$ variables, and
$\PGL(2)$ acts on it by virtue of its action on $\rM$. There is a subring, which we will call
$\mathcal R_{r,2}$, consisting of those $\PGL(2)$-invariant polynomial functions on
$\rM^r$.

The ring $\mathcal R_{r,2}$ is the coordinate ring of a universal categorical quotient scheme $\rM^r /\PGL(2)$ according
to \cite[Theorem 1.1]{Mumford1994}. A description of the rational points of $\Spec(\mathcal R_{r,2})$ is given by
\cite[12.6]{Artin1969}. In order to explain that description, we introduce some further terminology.

Let $F_r = \C \{ x_1, \dots, x_r\}$ denote the free associative algebra generated by $r$ indeterminates. Following
\cite{Artin1969}, given an $r$-tuple $\vec A =(A_1, \dots, A_r) \in \rM^r$, we define a matrix representation
\[ \phi_{\vec A} :F_r \to \rM \]
 by $x_i \mapsto A_i$. The set of representations is in obvious bijection with $\rM^r$. 

 Each representation $\phi_{\vec A}$ endows $\C^2$ with an $F_r$-module structure, and
 $\phi_{\vec A}$ is said to be \textit{semisimple} if $\C^2$ is totally reducible with this structure. That is,
 $\phi_{\vec A}$ is semisimple if either $\vec A$ has no common invariant subspace, in which case $\C^2$ is irreducible
 as an $F_r$-module, or if $\C^2$ decomposes as a direct sum of simultaneous eigenspaces of the $A_i$.

If $\phi_{\vec A}$ is not semisimple, there is an associated semisimple representation in the closure of the
$\PGL(2)$-orbit of $\phi_{\vec A}$. This
is a general fact proved in \cite[12.6]{Artin1969}. In the $2 \times 2$ case, the associated semisimple representation
is obtained from $\vec A$ by bringing $\vec A$ to simultaneous upper-triangular form
\[ \left(
    \begin{bmatrix}
      \lambda_1 & \ast \\ 0 & \mu_1 
    \end{bmatrix}, \begin{bmatrix}
      \lambda_2 & \ast \\ 0 & \mu_2 
    \end{bmatrix}, \dots, \begin{bmatrix}
      \lambda_r & \ast \\ 0 & \mu_r
    \end{bmatrix} \right) 
\] (possible since the $A_i$ do not generate $\rM$ and therefore they have a common eigenvector), then replacing the
$r$-tuple by
\[ \vec A'  = \left(
    \begin{bmatrix}
      \lambda_1 & 0 \\ 0 & \mu_1 
    \end{bmatrix}, \begin{bmatrix}
      \lambda_2 & 0 \\ 0 & \mu_2 
    \end{bmatrix}, \dots, \begin{bmatrix}
      \lambda_r & 0 \\ 0 & \mu_r
    \end{bmatrix} \right).
\]
We now can describe the points of $\Spec(\mathcal R_{r,2})$, as promised earlier.

\begin{theorem}[Artin] \label{th:artin}
  The rational points of $\Spec (\mathcal R_{r,2})$ are in one-to-one correspondence with the $\PGL(2)$-orbits of semisimple
  representations $\phi_{\vec A}$. The quotient map $q : \rM^r \to \Spec (\mathcal R_{r,2})$ takes a general $r$-tuple
  $\vec A'$ to the orbit of the associated semisimple representation.
\end{theorem}

\begin{proposition}  \label{pr:quotientIsPrinBundle}
  The quotient map $U(r) \to B(r) = U(r)/\PGL(2)$ is a principal $\PGL(2)$-bundle and $B(r)$ is a manifold.
\end{proposition}
\begin{proof}
  We show that $U(r) \to B(r)$ is a geometric quotient in the sense of \cite{Mumford1994}. We know that
  $\rM^r \to \Spec \mathcal R_{r,2}$ is a universal categorical quotient. Since $\rM^r$ admits only the trivial line
  bundle, and the trivial line bundle admits only one $\PGL(2)$-linearization, we know that every point of $\rM^r$ is
  semistable for the $\PGL(2)$ action.

  We observe that for all points $\vec A$ of $U(r)$, the hypotheses of \cite[Amplification 1.11(2)]{Mumford1994}
  apply. The regularity hypothesis is satisfied here because the action of $\PGL(2)$ on $U(r)$ is free. As for closure,
  if $\vec A \in U(r)$, then Theorem \ref{th:artin} assures us that the closed set $q^{-1}(q(\vec A))$ is precisely the orbit of
  $\vec A$.

  We deduce that $U(r)$ consists of stable points for the action of $\PGL(2)$, so that \cite[Converse 1.13]{Mumford1994}
  applies to say that the action of $\PGL(2)$ on $U(r)$ is proper and the quotient $B(r)$ is a quasiprojective
  $\C$-variety.
  
  Since the action of $\PGL(2)$ on $U(r)$ is proper, standard results in manifold theory  (see e.g., \cite{Lee2012}) tell us that $B(r)$ is a
  manifold, in this case a smooth variety, and the quotient map is a principal $\PGL(2)$-bundle map.
\end{proof}



  \section{Connectivity} \label{sec:connectivity}

  We make use of the following well-known result. The proof is an extended exercise in transversality.
  \begin{proposition}\label{pr:connComp} Let $Z \into M$ be a closed embedding of smooth manifolds of real codimension $d > 0$. Then
    the inclusion $M\sm Z \into M$ is $(d-1)$-connected.
  \end{proposition}
  \begin{corollary} \label{cor:connect}
    Let $Z \into M$ be a closed embedding of smooth manifolds of real codimension $d > 0$, and suppose $M$ is
    $(d-2)$-connected. Then $M \sm Z$ is $(d-2)$-connected
  \end{corollary}
  
  Here are two immediate applications of this.
  \begin{proposition}
    The space $U(r)$ is $(2r-4)$-connected.
  \end{proposition}
  \begin{proof}
    We know $U(r) = \rM^r \sm Z(r)$, where $Z(r)$ is defined as in Proposition . The variety $Z(r)$ is singular, but
    it admits a stratification into smooth locally closed subvarieties of complex dimension $\le 3r-1$. The space $\rM^r\homeo \C^{4r}$ is
    contractible, so by repeated application of Corollary \ref{cor:connect}, we see that $U(r)$ is $(d-2)$-connected where $d$ is
    the real codimension of $Z(r)$ in $\rM^r$, i.e., $d=2r-2$.
  \end{proof}

  \begin{corollary} \label{cor:connectivityOfClassMap}
    A map $B(r) \to B \PGL(2)$ classifying the principal $\PGL(2)$-bundle $U(r) \to B(r)$ is $(2r-3)$-connected.
  \end{corollary}
  \begin{proof}
    The map fits in a homotopy fibre sequence $U(r) \to B(r) \to B \PGL(2)$. Since $U(r)$ is $(2r-4)$-connected, the map
    $B(r) \to B \PGL(2)$ is $(2r-3)$-connected.
  \end{proof}

  \section{The case of \texorpdfstring{$r=2$}{r=2}} \label{sec:r=2}
  
  When $r=2$, we are able to determine the homotopy type of the variety $B(r)$ completely. In Section
  \ref{sec:InvTheory}, we observed that $B(r)$ is an open subvariety of an affine scheme $\Spec \mathcal
  R_{r,2}$. 
  \begin{theorem}[{\cite[Theorem 5]{Sibirskii1968}}]
    The ring $\mathcal R_{r,2}$ of $\PGL(2)$-invariant polynomials on the space of $r$-tuples $(A_1, \dots, A_r)$ of
    $2 \times 2$ matrices is generated by the elements
    \begin{align*}
      \Tr(A_i) & \quad \text{where $i  \in \{1, \dots, r\}$}, &  \Tr(A_i^2) & \quad \text{where $i  \in \{1, \dots, r\}$}, \\
      \Tr(A_iA_j) & \quad  \text{where $i, j \in\{1, \dots, r\}$ and $i < j$}, & \Tr(A_iA_j A_l)   & \quad \text{ where $i, j, l \in \{ 1, \dots, r\}$ and $i<j<l$}.
    \end{align*}
  \end{theorem}


  In fact, we will use this result only when $r=2$. Here, there are no triple products and so every element of the ring
  $\mathcal R_{2,2}$ is a polynomial in the functions
\begin{equation}
  \label{eq:5}
  \Tr(A_1), \Tr(A_2), \Tr(A_1^2), \Tr(A_2^2), \Tr(A_1A_2).
\end{equation}
Explicitly, this implies that $\Spec \mathcal R_{2,2}$ is a closed subscheme of the variety $\C^5$, and therefore that
$B(2)$ is a locally closed subscheme. We know that the dimension of $B(2)$ as a $\C$-variety is
$\dim U(r) - \dim \PGL(2) = 8 - 3=5$, so that $B(2)$ must actually be dense in $\C^5$ and
$\Spec \mathcal R_{2,2} = \C^5$. That is, the functions in \eqref{eq:5} are algebraically independent.

\begin{proposition} \label{pr:firstPoly}
  Let $A_1$ and $A_2$ be two $2\times 2$-matrices. Then $A_1$ and $A_2$ do not generate the matrix algebra
  $\rM$ if and only if $\Tr(A_1A_2)$ is a root of the equation
  \begin{equation}
    \label{eq:1}
    [2 x- \Tr(A_1)\Tr(A_2)]^2 = [2 \Tr(A_1^2)- \Tr(A_1)^2] [2 \Tr(A_2^2)- \Tr(A_2)^2].
  \end{equation}
\end{proposition}

This is a corollary of the results of \cite[Section 2]{Friedland1983}. Specifically, \cite[Theorem 2.9]{Friedland1983}
says that $A_1, A_2$ generate $\rM$, i.e., they are not simultaneously similar to upper triangular matrices, if and only
if they lie in the complement in $\rM^2$ of the closed affine variety defined by \eqref{eq:1}. This variety is called
$U$ in \cite[Theorem 2.2]{Friedland1983}.

As a consequence of this result, the space $B(2)$ is an open affine subvariety of $\C^5$ determined by the
non-vanishing of a single polynomial. 
In light of Proposition \ref{pr:tracelessHEquiv}, the space $B(2)$ admits a deformation retraction onto the space of
traceless pairs, $B'(2)$. When $\Tr(A_1) = 0 = \Tr(A_2)$, the equation \eqref{eq:1} simplifies, and so $B'(2)$ is the
affine open subvariety of $\C^3$ determined by the non-vanishing of the polynomial:
\begin{equation}
  \label{eq:2}
  x^2 - z_1z_2,
\end{equation}
where $z_1 = \Tr(A_1^2)$ and $z_2 = \Tr(A_2^2)$.

\begin{proposition} \label{pr:B2type}
  The space $B(2)$ is homotopy equivalent to the balanced product $S^1\times^{\ZZ/2\ZZ} S^2$, where $\ZZ/2\ZZ$ has the antipodal
  action on each factor.
\end{proposition}
\begin{proof}
  The inclusion $B'(2) \into B(2)$ is a homotopy equivalence, and $B'(2)$ is given by the non-vanishing of
  \eqref{eq:2}. After a change of coordinates, we can write $B'(2)$ as the affine complement of the subvariety of $\C^3$ determined by
  $x_1^2 + x_2^2 + x_3^2 = 0$.

  Consider the variety $Y$ consisting of quadruples $(\lambda, y_1, y_2, y_3) \in \C^4$ where $\lambda \in \C^\times$ and
  $y_1^2 + y_2^2 + y_3^2 = 1$.  This variety carries a free action by $\ZZ/2 \ZZ$ given by
  \[ (\lambda, y_1 , y_2, y_3) \mapsto (- \lambda, -y_1, -y_2, -y_3).\]
  There is a continuous map $f:\C^\times \times TS^2 \to B'(2)$ given by
  $f(\lambda, y_1, y_2, y_3) = (\lambda y_1, \lambda y_2, \lambda y_3)$. This map is surjective and satisfies
  $f(\lambda, y_1, y_2, y_3) = f(-\lambda, -y_1, -y_2, -y_3)$. It is elementary to check that the induced map
  \[\hat f : Y/(\ZZ/2\ZZ) \to B'(2)\] is a continuous bijection between $5$-manifolds, and therefore a
  homeomorphism.

  If $TS^2$ denotes the tangent bundle of the real $2$-manifold $S^2$, consisting of pairs of
  vectors $\vec u , \vec v \in \R^3$ satisfying $\vec u \cdot \vec v = 0$ and $\vec u \cdot \vec u = 1$, then the function
  \[ (\lambda, \vec u, \vec v) \mapsto (\lambda, \sqrt{1 + \| \vec v \|^2} \vec u + i \vec v ) \]
  is a homeomorphism $\C^\times \times  TS^2 \to Y$ (we learned this fact  from \cite{Huisman2016}). This homeomorphism
  is $\ZZ/2\ZZ$-equivariant where $\ZZ/2\ZZ$ acts by multiplication by $-1$ on each factor. We may write $\C^\times \times TS^2$ instead of $Y$.

  The embedding $S^1 \times S^2 \subset \C^\times \times TS^2$ sending $(\lambda, \vec u)$ to
  $(\lambda, \vec u, \vec 0)$ is a $\ZZ/2\ZZ$-equivariant weak equivalence. There is an induced homotopy equivalence
  $S^1 \times^{\ZZ/2\ZZ} S^2 \into B'(2)$.
\end{proof}

\section{Cohomology of \texorpdfstring{$U(r)$}{U(r)}} \label{sec:cohom}

In this section $\Hoh^*(X)$ will denote the cohomology of $X$ with rational
coefficients.

In the cases of $r>2$, we are unable to find an elegant description of the homotopy type of $B(r)$. We can,
however, calculate quite a good deal of its rational cohomology. The overall method is as follows. We first calculate
some of the rational cohomology of $U(r)$ by analyzing a stratification of $Z(r) = \rM^r \sm U(r)$ into smooth
subvarieties. This is achieved in this section. In principle, the calculation can be pushed further, but we do not need
any more for the rest of this paper.

In Section \ref{sec:rOdd}, we calculate some of the rational cohomology of $\Hoh^i(B(r))$ when $r$ is odd
by using the Serre spectral sequence of the homotopy fibre sequence $U(r) \to B(r) \to B \PGL(2)$. The difficulty is in
determining the differentials, and we employ several comparison arguments to show that the first differential that could
possibly be nonzero is, in fact, not zero. In Section \ref{sec:rEven}, we determine $\Hoh^i(B(r))$ for even $r$
in a range similar to that of the odd case, which we do by comparison to the odd case. This is enough to establish
Theorem \ref{th:fakeMainTheorem}.

We begin our calculation of the cohomology of $U(r)$ with some definitions:
\begin{itemize}
\item $T(r)$: those $r$-tuples $(A_1, \dots, A_r) \in \rM^r$ that pairwise commute, i.e., such that $[A_i,A_j] = \vec{0}$ for all $i,j \in \{1, \dots, r\}$. 
	\item $W(r) = Z(r) \sm T(r)$: those $r$-tuples $(A_1, \dots, A_r)$ that share an eigenvector and have the property that $[A_i,A_j] \neq \vec{0}$ for some $i,j \in \{1, \dots, r\}$.
\end{itemize} 
Note that $T(r) \subset Z(r)$, since the algebra generated by a pairwise commuting $r$-tuple is commutative. Moreover, $T(r)$ is a closed subvariety of $\rM^r$. 

\begin{proposition}\label{pr:dimT}
  The complex dimension of $T(r)$ is $2r+2$. 
\end{proposition}
\begin{proof}
  Let $L$ denote the subalgebra of $\rM$ consisting of scalar matrices. Consider the subset of $T(r)$ consisting of $r$-tuples of scalar matrices, $L^r$. This has complex
  dimension $r$.

  Now we work with the variety $T(r) \sm L^r$ of $r$-tuples of matrices that commute, but at least
  one of which is not scalar. Let $Y_i \subset T(r) \sm L^r$ denote the open subvariety where the $i$-th matrix, $A_i$,
  is not scalar. The sets $Y_i$ furnish a (Zariski) open cover of $T(r) \sm L^r$.

  Since $A_i$ is a $2\times 2$ matrix, its eigenvalues have geometric multiplicity $1$. Its characteristic polynomial
  agrees with its minimal polynomial, i.e., it is \textit{non-derogatory} in the
  terminology of \cite[3.2.4]{Horn1985}, and therefore a matrix commutes with $A_i$ if and only if it is of the form
  $aA_i +bI_2$.

  There is an isomorphism of varieties
  \[( \rM \sm L) \times (\C^{2})^{r-1} \homeo Y_i \] given by
  \[ A_i , (a_1, b_1, \dots, \widehat{a_i ,b_i} , \dots, a_r, b_r) \mapsto (a_1 A_i + b_1 I_2, \dots, A_i , \dots , a_r
    A_i + b_r I_2 ), \]
  where the hat denotes omission. It follows that the complex dimension of $Y_i$, and therefore of $T(r) \sm L^r$, is
  $4 + (2r-2)$, which is $2r+2$, as claimed.
      \end{proof}
      
	
      Next we consider the variety $W(r)$, which is dense in $Z(r)$\benw{the argument for this is elementary, but seems
        not to have been made.} and therefore of complex dimension $3r+1$. If an
      $r$-tuple $(A_1, \dots, A_r)$ is in $W(r)$, then the $A_i$ have a unique shared $1$-dimensional
      eigenspace. Consider the map
\begin{equation}
  \label{eq:7}
  p:W(r) \to \CP^1
\end{equation}
that sends an $r$-tuple in $W(r)$ to its common 1-dimensional eigenspace. This map is $\PGL(2)$-equivariant, where
$\PGL(2)$ acts on $\CP^1$ in the usual way. Since the action on the target is transitive, the fibres are all isomorphic
to the fibre $p^{-1}([1:0])$ consisting of $r$-tuples $(A_1, \dots, A_r)$ of upper triangular matrices such that
$[A_i,A_j] \neq 0$ for some $i,j$. With this in hand, we can construct algebraic local trivializations over the
standard open cover of $\CP^1$ to see that $p$ is a fibre bundle.
\begin{proposition} \label{pr:w}
	If $r > 2$, then
	\[
		\widetilde{\Hoh}^i(W(r);\Z) \cong 
		\begin{cases}
			\Z & \text{if } i = 2; \\
			0 & \text{if } i < 2r-3 \text{ and } i \neq 2.
		\end{cases}
	\]
\end{proposition}
\begin{proof}
  Identify the variety of $r$-tuples of upper triangular matrices with $\C^{3r}$. In the map of \eqref{eq:7}, the fibre
  $F = p^{-1}([1:0])$ is an open subvariety of $\C^{3r}$ and is therefore smooth. The space $W(r)$ is the total space of
  a fibre bundle in the category of $\C$-varieties with smooth base and fibre.
	
  We calculate the connectivity of $F$ as follows. The space $F$ consists of $r$-tuples of upper-triangular matrices
  $(A_1, \dots, A_r)$ such that at least one commutator $[A_i, A_j]$ is not $0$. Write $\C^{3r}$ for the
  space of all $r$-tuples of upper triangular matrices, and let $Q$ denote the complement of $F$ in $\C^{3r}$, i.e., the
  closed subvariety determined by the vanishing of the commutators. An argument very similar to that of Proposition
  \ref{pr:dimT} shows that $\dim_\C Q = 2r+1$, so that the codimension of $Q$ in $UT$ is $r-1$. Since $F$ is the
  complement in $\C^{3r}$ of a closed subvariety of complex codimension $r-1$, the variety $F$ is $2r-4$-connected by
  Corollary \ref{cor:connect}.

  The statement of the proposition now follows from the Serre spectral sequence in rational cohomology for the fibration
  $F \to W(r) \to \CP^1$.
\end{proof}
\begin{proposition} \label{pr:cohoU}
	If $r > 2$, then
	\[
		\widetilde{\Hoh}^i(U(r);\Z) \iso 
		\begin{cases}
			\alpha_i \Z & \text{if } i = 2r-3 \text{ or } i=2r-1; \\
			0 & \text{if } i < 4r-6 \text{ and } i \neq 2r-3,2r-1.
		\end{cases}
	\]
\end{proposition}
\begin{proof}
  Corollary \ref{cor:connect} and Proposition \ref{pr:dimT} combine to show $\rM^r \sm T(r)$ is $(4r-6)$-connected. It
  is $4r$-dimensional and $W(r)$ is therefore of complex codimension $r-1$ in it. Now applying the Thom--Gysin sequence
  \[
    \begin{tikzcd}
      \cdots \rar & \widetilde\Hoh^{i-1}(U(r)) \rar & \Hoh^{i-2(r-1)}(W(r)) \rar & \widetilde \Hoh^i(\rM^r \sm T(r))
      \rar & \widetilde \Hoh^i(U(r)) \rar & \cdots 
    \end{tikzcd}
  \] 
  to $\rM^r \sm T(r)$, $W(r)$ and $U(r) = (\rM^r \sm T(r) ) \sm W(r)$ gives the result.
\end{proof}

\section{The case of odd \texorpdfstring{$r$}{r}}\label{sec:rOdd}

Again, in this section $\Hoh^*(X)$ denotes the cohomology of $X$ with rational coefficients.

We compute $\Hoh^i(B(r))$ in degrees $i < 4r-6$ by means of the Serre spectral sequence associated to
the homotopy fibre sequence $U(r) \to B(r) \to B \PGL(2)$:
\[
	\Eoh_2^{p,q} = \Hoh^p(B \PGL(2); \Hoh^q(U(r))) \Rightarrow \Hoh^{p+q}(B(r)).
\]
We restrict our attention to terms $\Eoh_k^{p,q}$ with $p+q < 4r-6$. We may identify
\[ \Hoh^*(B \PGL(2;\C)) = \Hoh^*(B \SO_3)\] since $\SO_3 \weq \SO(3; \C) \iso \PGL(2;\C)$. The rational cohomology
ring $\Hoh^*(B \SO_3)$ is $\Q[p_1]$, where $p_1$ is the Pontryagin class and has degree 4, \cite[Theorem
1.4]{Brown1982}. The $\Eoh_2$-page of the spectral sequence can be determined in the range $p+q < 4r-6$ by reference to
Proposition \ref{pr:cohoU}.
\smallskip

Suppose now that $r$ is odd. The first differential that is not obviously $0$ is the
transgressive differential $d_{2r-2}: \Hoh^{2r-3}(U(r)) \to \Hoh^{2r-2}(B \PGL(2))$, indicated in Figure \ref{fig:E2r-2SSSB(r)}.
\begin{figure}[h]
	\centering
        \includegraphics{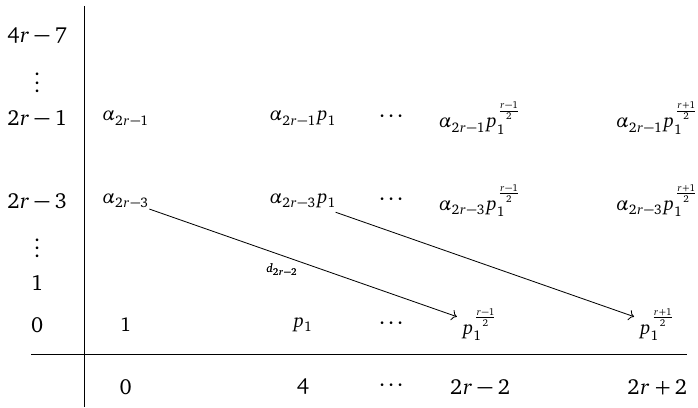}
	\caption{The $\Eoh_{2r-2}$-page of the Serre spectral sequence converging to $\Hoh^*(B(r))$ when $r$ is odd. All indicated classes generate a term isomorphic to $\Q$. The empty terms are all $0$.}
	\label{fig:E2r-2SSSB(r)}
\end{figure}

Once we establish that this differential is not $0$, then the computation of $\Hoh^i(B(r))$ in the range $i < 4r-6$
(Proposition \ref{pr:cohoBOdd}) follows. The differentials with source and target in the range $p+q < 4r-6$ on
succeeding pages all vanish for trivial reasons.

\subsection*{The first comparison}
Embed $S^1$ in $\PGL(2)$ via the map
\[
	\lambda \mapsto 
	\begin{bmatrix}
		1 & 0 \\
		0 & \lambda
	\end{bmatrix}.
\]
Then $\lambda \in S^1$ acts on $U(r)$ by
\begin{equation}\label{eq:S1Action}
	\left( 
	\begin{bmatrix}
		a_1 & b_1 \\
		c_1 & d_1
	\end{bmatrix}, \dots, 
	\begin{bmatrix}
		a_r & b_r \\
		c_r & d_r
	\end{bmatrix}
	\right) \mapsto 
	\left(
	\begin{bmatrix}
		a_1 & \bar{\lambda} b_1 \\
		\lambda c_1 & d_1
	\end{bmatrix}, \dots, 
	\begin{bmatrix}
		a_r & \bar{\lambda} b_r \\
		\lambda c_r & d_r
	\end{bmatrix}
	\right).
\end{equation}
We obtain a map of homotopy fibre sequences
\begin{equation}\label{eq:1stComp}
	\begin{tikzcd}
		U(r) \rar & B(r) \rar & B \PGL(2) \\
		U(r) \rar \uar[equal] & U(r)/S^1 \rar \uar & B S^1. \uar
	\end{tikzcd}
      \end{equation}
      
\subsection*{The second comparison}
Consider the inclusion $i:S^{2r-3} \times S^{2r-3} \to U(r)$ given by
\[
	((b_1, \dots, b_{r-1}),(c_1, \dots, c_{r-1})) \mapsto 
	\left(
	\begin{bmatrix}
		0 & b_1 \\
		c_1 & 0
	\end{bmatrix}, \dots, 
	\begin{bmatrix}
		0 & b_{r-1} \\
		c_{r-1} & 0
	\end{bmatrix},
	\begin{bmatrix}
		1 & 0 \\
		0 & -1
	\end{bmatrix} 
	\right),
\]
where $S^{2r-3}$  is embedded in $\C^{r-1}$ in the usual way. To see that the target of $i$ is in fact $U(r)$, observe that the eigenspaces of the last matrix of an $r$-tuple $(A_1, \dots, A_r)$ in the image of $i$ are $\spn \{e_1\}$ and $\spn \{e_2\}$. Since $b_i,c_j \neq 0$ for some $i,j \in \{1, \dots, r-1\}$, the $r$-tuple $(A_1, \dots, A_r)$ does not have a common eigenvector and hence generates the matrix algebra.

Endow $S^{2r-3} \times S^{2r-3}$ with the $S^1$-action $\lambda\cdot(b,c) = (\bar{\lambda} b, \lambda c)$. Then $i$ is
equivariant with respect to this action, so we obtain a homotopy commutative diagram of homotopy fibre sequences:
\begin{equation}\label{eq:2ndComp}
	\begin{tikzcd}
		U(r) \rar & U(r)/S^1 \rar & B S^1 \\
		S^{2r-3} \times S^{2r-3} \uar["i"] \rar & S^{2r-3} \times_{S^1} S^{2r-3} \uar \rar & BS^1. \uar[equal]
	\end{tikzcd}
\end{equation}
\begin{notation}
	Let $'\Eoh_k^{p,q}$, $'d_k$ (respectively $''\Eoh_k^{p,q}$, $''d_k$) denote the terms and differentials of the Serre spectral sequence associated to the top (bottom, respectively) homotopy fibre sequence in \eqref{eq:2ndComp}. We also make the identification 
	\[
		\Hoh^{2r-3}(S^{2r-3} \times S^{2r-3}) \cong \Hoh^{2r-3}(S^{2r-3}) \oplus \Hoh^{2r-3}(S^{2r-3})
	\] 
	via the K\"{u}nneth formula.
\end{notation}
\begin{lemma}\label{lem:''d_{2r-2}}
	The transgressive differential $''d_{2r-2}: \Hoh^{2r-3}(S^{2r-3} \times S^{2r-3}) \to \Hoh^{2r-2}(B S^1)$ is
        surjective and the kernel is $\spn \{(\rho_{2r-3}, -\rho_{2r-3})\}$, where $\rho_{2r-3}$ is a generator of $\Hoh^{2r-3}(S^{2r-3})$. 
\end{lemma}
\begin{proof}
  Comparison to the Serre spectral sequence associated to $S^{2r-3} \to \CP^{r-2} \to B S^1$ via either projection shows
  that the differential in question is surjective. For the second part, it suffices to show that
  $''d_{2r-2}(\beta,\gamma) = {''d_{2r-2}}(\gamma,\beta)$ for classes $\beta, \gamma \in \Hoh^{2r-3}(S^{2r-3})$. The
  product $S^{2r-3} \times S^{2r-3}$ is equipped with an $S^1$-equivariant involution
  $\sigma: (b,c) \mapsto (\bar{c},\bar{b})$, where $\bar{a}$ denotes the $(r-1)$-tuple
  $(\bar{a}_1, \dots, \bar{a}_{r-1})$. The map $\sigma$ induces a self-map of spectral sequences
  $\sigma_k^{p,q}: {''\Eoh_k^{p,q}} \to {''\Eoh_k^{p,q}}$. In particular, the differential $''d_{2r-2}$ is invariant
  under the action of $\sigma$. Note that the homeomorphism $S^{2r-3} \to S^{2r-3}$ defined by $a \mapsto \bar{a}$ has
  degree 1, being the composition of an even number of reflections. Hence $\sigma$ is homotopic to the
  self-homeomorphism of $S^{2r-3} \times S^{2r-3}$ that switches factors. The induced map $\sigma^*$ on rational
  cohomology in degree $2r-3$ also switches factors. The result follows.
\end{proof}
Next we discuss the induced map $i^*: \Hoh^*(U(r)) \to \Hoh^*(S^{2r-3} \times S^{2r-3})$. 
\begin{lemma} \label{lem:inonzero}
	The class $i^*(\alpha_{2r-3})$ is nonzero. 
\end{lemma}
\begin{proof}
	Consider the inclusion $j: \C^{r-1} \mapsto \rM^r \sm T(r)$ given by
	\[
		(b_1, \dots, b_{r-1}) \mapsto 
		\left(
		\begin{bmatrix}
			0 & b_1 \\
			1 & 0
		\end{bmatrix}, \dots , 
		\begin{bmatrix}
			0 & b_{r-1} \\
			1 & 0
		\end{bmatrix},
		\begin{bmatrix}
			1 & 0 \\
			0 & -1
		\end{bmatrix}
		\right).
	\]
	There is a pullback square
	\[
		\begin{tikzcd}
			\{\vec{0} \} \rar \dar \ar[dr,phantom,"\lrcorner",very near start] & \C^{r-1} \dar["j"] \\
			W(r) \rar & \rM^r \sm T(r)
		\end{tikzcd}
	\]
	where the two horizontal maps are closed inclusions of smooth $\C$-varieties. The map $j$ induces a map of Thom--Gysin sequences provided $W(r)$ is transverse to $j$. 
	
	To see that the intersection is transverse, let $\epsilon > 0$ and $\alpha: (-\epsilon, \epsilon) \to W(r)$ be a smooth path such that $\alpha(0) = j(\vec{0})$. We may write 
	\[
		\alpha(t) = 
		\left(
		\begin{bmatrix}
			a_1(t) & b_1(t) \\
			1 + c_1(t) & d_1(t)
		\end{bmatrix}, \dots, 
		\begin{bmatrix}
			a_{r-1}(t) & b_{r-1}(t) \\
			1 + c_{r-1}(t) & d_{r-1}(t) 
		\end{bmatrix},
		\begin{bmatrix}
			a_r(t) + 1 & b_r(t) \\
			c_r(t) & d_r (t) -1
		\end{bmatrix}
		\right)
	\]
	for some smooth paths $a_k, b_k, c_k, d_k: (-\epsilon, \epsilon) \to \C$ that evaluate to $0$ at $t=0$. We wish
        to show that $b_k'(0) = 0$ for each $k = 1, \dots, r-1$. Recall there is a map $p: W(r) \to \CP^1$ that sends an $r$-tuple to
        its common eigenspace. Since $p\alpha(0) = [0:1]$, we may assume that the image of $\alpha$ lies in
        $p^{-1}(U_1)$ where $U_1 = \{[z:1]: z \in \C\} \homeo \C$. Write $p\alpha(t) = [z(t):1]$, and let $k \in \{1, \dots, r-1\}$. If the $k$th matrix of
        $\alpha(t)$ has the eigenspace $[z(t):1]$, then
	\[
		a_k(t)z(t) + b_k(t) = (1 + c_k(t)) z(t)^2 + d_k(t) z(t).
	\]
        Taking derivatives with respect to $t$ and evaluating at $t=0$, we get $b_k'(0) = 0$.
	
	A portion of the induced map of Thom--Gysin sequences takes the form:
	{\small\[
		\begin{tikzcd}
		\cdots \rar& 	0 \rar & \widetilde \Hoh^{2r-3}(\C^{r-1} \sm \vec{0}) \rar{\iso} & \Hoh^0(\{\vec{0}\}) \rar & 0 \rar & \cdots\\
			\cdots \rar & \widetilde\Hoh^{2r-3}(\rM^r \sm T(r)) \rar \uar["j^*"] & \widetilde \Hoh^{2r-3}(U(r)) \rar["\partial"]
                        \uar["j^*"] &  \Hoh^0(W(r)) \rar \uar["j^*"] &  \widetilde \Hoh^{2r-2}(\rM^r \sm T(r)) \uar["j^*"] \rar & \cdots .
		\end{tikzcd}
	\]}
	Since $\rM^r \sm T(r)$ is $(4r-6)$-connected by Corollary \ref{cor:connect} and Proposition \ref{pr:dimT}, the
        map $\partial$ is an isomorphism. It follows that $j^*: \Hoh^{2r-3}(U(r)) \to \Hoh^{2r-3}(\C^{r-1} \sm \vec{0})$
        is an isomorphism. The map $j$ restricted to $\C^{r-1} \sm \vec{0}$ is, up to homotopy, the restriction of the
        map $i$ to one of the spherical factors. It follows that $i^*$ is nontrivial in degree $2r-3$. 
\end{proof}
\begin{lemma}\label{lem:symOfi*}
  There exists a nonzero element $\beta$ of $\Hoh^{2r-3}(S^{2r-3})$ such that $i^*(\alpha_{2r-3}) =(\beta,\beta)$.
\end{lemma}
\begin{proof}
  Set $i^*(\alpha_{2r-3}) = (\beta, \gamma)$. In light of Lemma \ref{lem:inonzero}, it suffices to show
  $\beta = \gamma$.

  The space $U(r)$ admits an action by the connected Lie group $\C^\times \times \PGL(2)$, given by
  \[ (\lambda, P) \cdot (A_1, \dots, A_r) = (\lambda PA_1P^{-1}, \dots ,\lambda PA_rP^{-1}).\] Let
  $\tau : U(r) \to U(r)$ denote the action by $\left(-1, \begin{bmatrix} 0 & -1 \\ 1
      &0  \end{bmatrix}\right)$. Connectivity implies that $\tau \weq \id_{U(r)}$.
  
  Explicitly $\tau$ acts in the following way on each entry of $\vec A$:
  \[
    \begin{bmatrix}
        a & b \\
        c & d
      \end{bmatrix} \mapsto
    \begin{bmatrix}
      -d & c \\
      b & -a
    \end{bmatrix},
  \]
  and we observe that the diagram
  \[
    \begin{tikzcd}
      S^{2r-3} \times S^{2r-3} \dar["i"] \rar["\text{swap}"] & S^{2r-3} \times S^{2r-3} \dar["i"] \\
      U(r) \rar["\tau"] & U(r)
    \end{tikzcd}
  \]
  commutes. The result follows.
\end{proof}
\begin{proposition}\label{pr:diffrOdd}
	When $r>2$ is odd, the class $d_{2r-2}(\alpha_{2r-3})$ is a generator of $\Hoh^{2r-2}(B \PGL(2))$.
\end{proposition}
\begin{proof}
	Naturality of the Serre spectral sequence for the second comparison says that $'d_{2r-2} = {''d_{2r-2}}i^*$. The latter map is an isomorphism by Lemmas \ref{lem:''d_{2r-2}} and \ref{lem:symOfi*}. The first comparison then shows that $d_{2r-2}$ is an isomorphism. 
\end{proof}
\begin{proposition}\label{pr:cohoBOdd}
	If $r > 2$ is odd, 
	\[
		\Hoh^i(B(r)) \cong 
		\begin{cases}
			\Q 	& \text{if } i \leq 2r-6 \text{ and } i \equiv 0 \pmod{4}; \\
			\Q 	& \text{if } 2r-1 \leq i < 4r-6 \text{ and } i \equiv 1 \pmod{4}; \\
			0 	& \text{otherwise when } i < 4r-6.
		\end{cases}
	\]
\end{proposition}
\begin{proof}
  In the range $p+q < 4r-6$, one has $\Eoh_{2r-1}^{p,q} \cong \Eoh_\infty^{p,q}$. The result follows immediately.
\end{proof}

\section{The case of even \texorpdfstring{$r$}{r} and the proof of Theorem \ref{th:fakeMainTheorem}}\label{sec:rEven}

\begin{proposition}\label{pr:cohoBEven}
	If $r \geq 2$ is even, 
	\[
		\Hoh^i(B(r)) \cong
		\begin{cases}
			\Q	& \text{if } i \leq 2r-4 \text{ and } i \equiv 0 \pmod{4};\\
			\Q 	& \text{if } 2r-3 \leq i < 4r-6 \text{ and } i \equiv 1 \pmod{4}; \\
			0	& \text{otherwise when } i < 4r-6.
		\end{cases}
	\]
\end{proposition}
\begin{proof}
	The first possibly supported differential in this case is the transgressive differential \[d_{2r}: \Hoh^{2r-1}(U(r)) \to \Hoh^{2r}(B \PGL(2)).\] There is a $\PGL(2)$-equivariant inclusion $U(r) \to U(r+1)$ given by
	\[
		(A_1, \dots, A_r) \mapsto (A_1, \dots, A_r,0).
	\]
	Comparison to the Serre spectral sequence associated to $U(r+1) \to B(r+1) \to B \PGL(2)$ together with Proposition \ref{pr:diffrOdd} shows that $d_{2r}(\alpha_{2r-1})$ is generator of $\Hoh^{2r}(B \PGL(2))$. Similar to the odd case, one has $\Eoh_{2r+1}^{p,q} \cong \Eoh_\infty^{p,q}$ in the range $p+q < 4r-6$. 
\end{proof}

\begin{proof}[Proof of Theorem \ref{th:fakeMainTheorem}]
  Set $i$ to be the largest integer such that $i \equiv 0 \pmod 4$ and $i \le d$, which is to say $i = 4 \lfloor d/4
  \rfloor$. Set $r =i/2 + 1$ so that
  \[ r = 2\left\lfloor \frac{d}{4} \right\rfloor + 1. \]
  We observe that  
  \[
    \begin{tikzcd}
      \Hoh^i(B(r) ; \ZZ) \arrow[r] & \Hoh^i(B \PGL(2) ; \ZZ)
    \end{tikzcd}
  \]
  is not injective, by noting that $i/2$ is even and so $r$ is odd, then using Proposition \ref{pr:cohoBOdd} with
  $i=2r-2$. Therefore \cite[Theorem 11.4]{First2022} applies with $\rM = \Mat_{2 \times 2}(\C)$ as the underlying algebra (note that our $\Hoh^i(B(r); \ZZ)$ is
  $\Hoh^i_G(U_r)$ in \cite[Theorem 11.4]{First2022}). This assures us that there exists a finite type $\CC$-ring $R'$
  having Krull dimension $i$ and a degree-$2$ Azumaya algebra algebra $A'$ over $R'$ such that $A'$ cannot be generated by $r$ elements.

  By construction, $i\le d < i+4$. If we replace $R'$ by a polynomial ring $R$ in $d-i$ variables $\{x_j\}_{j=1}^{d-i}$
  over $R'$, we obtain a ring having Krull dimension $d$ and the modified algebra $A = A' \tensor_{R'}R$ similarly
  cannot be generated by $r$ elements, or else the specialization of those generators at $x_j=0$ would serve as
  generators for $A'$.

  Therefore we have produced a degree-$2$ Azumaya algebra $A$ over a ring $R$ of Krull dimension $d$ such that $A$ cannot be
  generated by fewer than
  \[ 2 \left \lfloor \frac d 4 \right \rfloor + 2\]
  elements.  
\end{proof}

\printbibliography

 \end{document}
